\documentclass[12pt]{amsart} %leqno is the option to put formula numbers on the left side
\usepackage{amsmath, amssymb}
\usepackage{amsthm} %theorem environment option
\usepackage{color}
\usepackage{amscd,amsfonts,amsxtra,amstext,mathtools,latexsym,bbm,hyperref,enumitem,indentfirst,url}
\usepackage{graphicx}
\theoremstyle{plain} %text of this environment is typesetted in italics
\newtheorem{theorem}{\indent\sc Theorem}[section]
\newtheorem{lemma}[theorem]{\indent\sc Lemma}

\newtheorem{proposition}[theorem]{\indent\sc Proposition}

\theoremstyle{definition} %text of this environment is typesetted in roman letters
\newtheorem{definition}[theorem]{\indent\sc Definition}
\newtheorem{remark}[theorem]{\indent\sc Remark}

\newcommand{\N}{\mathbb{N}}
\newcommand{\Z}{\mathbb{Z}}
\newcommand{\R}{\mathbb{R}}
\newcommand{\Sp}{\mathbb{S}}
\newcommand{\C}{\mathbb{C}}
\newcommand{\Hyp}{\mathbb{H}}
\newcommand{\B}{\mathbb{B}}
\newcommand{\bigO}{\mathcal{O}}
\newcommand{\su}{\mathfrak{su}_2}
\newcommand{\sll}{\mathfrak{sl}_2}

\newcommand{\GL}{\mathrm{GL}_2}
\newcommand{\SL}{\mathrm{SL}_2}
\newcommand{\SU}{\mathrm{SU}_2}

\newcommand{\GU}{\mathrm{GU}_2}
\newcommand{\CMC}{$\mathrm{CMC}\,$}
\newcommand{\ODE}{$\mathrm{ODE}\;$}
\newcommand{\dd}{\mathrm{d}}
\newcommand{\identity}{\mathbbm{1}}
\DeclareMathOperator{\diag}{diag}
\DeclareMathOperator{\sign}{sign}
\DeclareMathOperator{\fix}{fix}
\DeclareMathSymbol{\Nu}{\mathalpha}{operators}{"4E}
\newcommand{\ph}[1]{\left[#1\right]}
\newcommand{\lc}[1]{\left[\left[#1\right]\right]}
\newcommand{\out}[1]{\setminus\lbrace#1\rbrace}
\DeclareMathAlphabet{\mathcal}{OMS}{cmsy}{m}{n}
\SetMathAlphabet{\mathcal}{bold}{OMS}{cmsy}{b}{n}

\hypersetup{
    bookmarks=true,         
    unicode=false,          
    pdftoolbar=true,        
    pdfmenubar=true,        
    pdffitwindow=false,     
    pdfstartview={FitH},    
    pdfnewwindow=true,      
    colorlinks=false,       
    linkcolor=red,          
    citecolor=green,        
    filecolor=magenta,      
    urlcolor=cyan           
}

\usepackage[noabbrev]{cleveref}
\begin{document}

\title[CMC Trinoids with one Irregular End]{Constant Mean Curvature Trinoids with one Irregular End} %title of paper and the running head option

\author[M. Kilian]{Martin Kilian} %first author's name and the running head option

\author[E. Mota]{Eduardo Mota} %second author's name and the running head option

\author[N. Schmitt]{Nicholas Schmitt} %third author's name and the running head option

%\dedicatory{Dedicated to Professor Xxx Yyy on his sixtieth birthday}

%%%%%%%%%%%%%%% footnote %%%%%%%%%%%%%%%%
%In case \subjclass[2010] command is not effective
%(or the version of amsart.cls is old), write as follows instead:
%\renewcommand{\thefootnote}{\fnsymbol{footnote}}
%\footnote[0]{2010\textit{ Mathematics Subject Classification}.
%Primary 00; Secondary 00.}
%
%%%%%%%%%%%% Authors' addresses %%%%%%%%%%%%%
\address{% First Author
Department of Mathematics, University College Cork, Ireland
}
\email{m.kilian@ucc.ie}

\address{% Second Author
Department of Mathematics, University College Cork, Ireland
}
\email{eduardomotasanchez@ucc.ie}

\address{% Third Author 
Institut f{\''u}r Mathematik, TU Berlin, Germany
}
\email{schmitt@math.tu-berlin.de}

\maketitle

\begin{abstract}
	We construct a new five parameter family of constant mean curvature trinoids with two asymptotically Delaunay ends and one irregular end.
\end{abstract}

%
%%%%%%%%%%%%%%%%%%%%%%%%%%%%%%%%%%%%%%%%%%%%%%%%%%%%%%%%%%%%%%%%%%%%%%%%%

  \section*{Introduction}

    The generalised Weierstrass representation \cite{DPW} constructs conformal constant mean curvature (CMC) immersions in Euclidean $3$-space from a holomorphic 1-form $\xi$ on a Riemann surface $\Sigma$. The associated period problem involves showing that the monodromy group of $\xi$ is pointwise unitarisable along the unit circle of the \emph{spectral parameter} $\lambda$. For the case that $\Sigma$ is the thrice-punctured Riemann sphere and $\xi$ has three regular singular points at the punctures, the resulting three-parameter family of \CMC trinoids have asymptotically Delaunay ends \cite{KilKRS}. In this paper we extend this result by replacing one of the regular singularities with an irregular singularity of rank $1$. The corresponding second-order scalar \ODE is the confluent Heun equation (CHE) with five free parameters. The monodromy can be computed by an asymptotic formula for the connection matrix between solutions at the two regular singular points given by Sch\"afke-Schmidt \cite{ScSc}. In this way, we construct a new five-parameter family of \CMC trinoids with two Delaunay ends and one irregular end.\\

%\textbf{Acknowledgements.} The authors are grateful to an anonymous referee for pointing out faults and corrections.

  %
  %%%%%%%%%%%%%%%%%%%%%%%%%%%%%%%%%%%%%%%%%%%%%%%%%%%%%%%%%%%%%%%%%%%%%%%%%

  \section{The Generalised Weierstrass Representation and the Monodromy Problem}\label{sec:dpwmethod}

    Let us briefly recall the generalized Weierstrass representation \cite{DPW} to set the notation and conventions adapted from \cite{KilKRS}. It consists of the following three steps:

    (i) On a connected Riemann surface $\Sigma$, let $\xi$ be a holomorphic 1-form, called a \emph{potential}, with values in the loop algebra of maps $\Sp^1\to\sll(\C)$. The potential $\xi$ has a simple pole in its upper right entry in the loop parameter $\lambda$ at $\lambda=0$, and has no other poles in the unit $\lambda$ disk. Moreover, the upper-right entry of $\xi$ is non-zero on $\Sigma$.  Let $\Phi$ be a solution of
    \begin{equation} \label{eq:IVP}
      \dd\Phi = \Phi\,\xi\; .
    \end{equation}

    (ii) Let $\Phi=F\,B$ be the pointwise Iwasawa factorization on the universal cover $\widetilde\Sigma$.

    (iii) Then $f=F'F^{-1}$ is an associated family of conformal \CMC immersions $\widetilde\Sigma \to \su\cong\R^3$. The prime denotes differentiation with respect to $\theta$, where $\lambda=e^{i\theta}$. 
      
    The gauge action is defined by $\xi.g:=g^{-1}\xi g+g^{-1}\dd g$. If $\dd\Phi = \Phi\,\xi$, then $\Phi g$ solves $\dd\Psi=\Psi(\xi.g)$. If $g$ is a map on $\Sigma$ with values in a positive loop group of $\SL \C$, then $\xi.g$ is again a potential.

    Now let $\Sigma = \mathbb{C}\out{z_0, z_1, z_\infty}$ be the thrice-punctured Riemann sphere with punctures $z_0 = 0,\,z_1 = 1$ and $z_\infty = \infty$. Let $\Delta$ denote the group of deck transformations of the universal cover $\tilde{\Sigma}$. Let $\xi$ be a holomorphic potential on $\Sigma$. Then $\xi(\tau(z),\lambda) = \xi(z,\lambda)$ for all $\tau \in \Delta$. Let $\Phi$ be a solution of $\dd\Phi = \Phi \xi$. Note that in general $\Phi$ it is only defined on $\tilde{\Sigma}$. We define the \emph{monodromy} $M_\tau$ with respect to $\tau$ by
    \begin{equation}\label{eq:monodromygeneral}
      M_\tau = \Phi(\tau(z),\lambda) \, \Phi(z,\lambda)^{-1}.
    \end{equation}
    The period problem cannot be solved simultaneously for the whole associated family, so we contend ourselves with solving it for the member of the associated family $\lambda = 1$. For $i\in\lbrace 0,\,1,\,\infty\rbrace$ let $\gamma_i$ be a loop around the puncture $z_i$ and $M_i$ the monodromy of $\Phi$ along this loop. The period problem consists of the following three conditions \cite{KilKRS}: For all $i\in\lbrace 0,\,1,\,\infty\rbrace$
    \begin{equation}\label{eq:monodromyproblem}
      \left\{
      \begin{array}{lll}
        M_i&\text{takes values in the unitary loop group},\\
        \left. {M_i}\right|_{\lambda = 1} &= \pm \rm{Id},\\
        \left. \partial_\lambda {M_i} \right|_{\lambda = 1} &= 0 \,.
      \end{array}
      \right.
    \end{equation}
    %
  %
  %%%%%%%%%%%%%%%%%%%%%%%%%%%%%%%%%%%%%%%%%%%%%%%%%%%%%%%%%%%%%%%%%%%%%%%%%

  \section{The Associated Second Order \ODE}\label{sec:associatedode}
    Consider a holomorphic potential 
    \begin{equation}\label{eq:0pot}
      \eta_0=\begin{pmatrix} S(z,\lambda) & R(z,\lambda)\\ T(z,\lambda) & -S(z,\lambda)\end{pmatrix}\;dz
    \end{equation}
    on $\Sigma=\C\out{z_0, z_1, z_\infty}$. Set
    \begin{equation}\label{eq:0gauge}
      g=\begin{pmatrix} e^{-f(z,\lambda)} & 0\\ 0 & e^{f(z,\lambda)}\end{pmatrix},\quad\text{where}\;f(z,\lambda)=\int_{z_*}^z S(w,\lambda)\; dw.
    \end{equation}
    Then $g$ is a positive loop and gauging the holomorphic potential $\eta_0$ by $g$ yields
    \begin{equation}\label{eq:1pot}
      \eta_1:=\eta_0.g=\begin{pmatrix} 0 & e^{2 f}R(z,\lambda)\\ e^{-2 f}T(z,\lambda) & 0\end{pmatrix}\;dz.
    \end{equation}
    Hence we can assume without loss of generality that our potential is off-diagonal. Let us then consider a potential, denoted $\eta$, of the form
    \begin{equation}\label{eq:offdpot}
      \eta=\begin{pmatrix} 0 & \nu(z,\lambda)\\ \rho(z,\lambda) & 0\end{pmatrix}\;dz.
    \end{equation}
    For the purpose of our work it will be convenient to work with the associated scalar second order \ODE corresponding to the $2 \times 2$ system (\ref{eq:IVP}), given by the following straightforward
    \begin{lemma} \label{associatedode}
      Solutions of $\dd\Phi=\Phi\eta$ are of the form
      \begin{equation}\label{eq:generalsolution}
      \begin{pmatrix} y_1'/\nu & y_1 \\ y_2'/\nu & y_2 \end{pmatrix}
      \end{equation}
      where $y_1$ and $y_2$ are a fundamental system of the scalar \ODE
      \begin{equation}\label{eq:scalar2ndode}
        y'' - \frac{\nu'}{\nu}\, y' - \rho\,\nu\,y = 0.
      \end{equation}
    \end{lemma}
    %
  %
  %%%%%%%%%%%%%%%%%%%%%%%%%%%%%%%%%%%%%%%%%%%%%%%%%%%%%%%%%%%%%%%%%

  \section{Two Regular Singular Points}\label{sec:regularpoints}

    Our goal is to construct \CMC trinoids for which two ends are regular and one end is irregular.  We will assume that at the two ends $z_0$ and $z_1$, the potential is a holomorphic perturbation of a Delaunay potential, and hence regular singular there. In addition, our potentials will have a singularity of rank 1  at $z_\infty$, making it an irregular end. These choices will determine the form of the associated scalar \ODE (\ref{eq:scalar2ndode}).

    Note that in this and the following sections, we omit the dependence on the spectral parameter $\lambda$. Let $\vartheta_0, \vartheta_1\in\C\setminus\tfrac{1}{2}\Z$ be parameters, and $a$ and $b$ functions on $\Sigma$ such that $a$ is holomorphic at $z_0$ and $z_1$, and $b$ is allowed to have simple poles at $z_0$ and $z_1$. Define
    \begin{align*}\label{eq:potdoublepoles}
      \xi&:=\begin{pmatrix} 0 & 1\\ Q(z) & 0\end{pmatrix}\;dz,\\
      Q(z)&:=\frac{\vartheta_0(\vartheta_0-1)}{(z-z_0)^2}+\frac{\vartheta_1(\vartheta_1-1)}{(z-z_1)^2}+b(z).
    \end{align*}
    The points $z_0$ and $z_1$ are regular singular points of $\xi$. These double poles can be gauged to simple poles by
    \begin{equation}\label{eq:gaugepoles}
      g_0:=\begin{pmatrix} 1 & 0\\ G & 1\end{pmatrix},\quad\text{where}\;G(z)=\frac{\vartheta_0}{z-z_0}+\frac{\vartheta_1}{z-z_1}+a(z),
    \end{equation}
    to obtain
    \begin{equation}\label{eq:potregpoints}
      \eta:=\xi.\left(g_0^{-1}\right)=A_0\frac{dz}{z-z_0}+A_1\frac{dz}{z-z_1}+B\;dz
    \end{equation}
    where $B$ is holomorphic at $z_0$ and $z_1$ and
    \begin{equation}\label{eq:matrices}
      A_k:=H_k \Delta_k H_k^{-1},\quad \Delta_k:=\diag(-\vartheta_k,\vartheta_k),\quad H_k:=\begin{pmatrix} 1 & 0\\ * & 1\end{pmatrix},\quad k\in\lbrace 0, 1\rbrace.
    \end{equation}
    It is only left to see that a general potential like in (\ref{eq:0pot}), the one we started with, is equivalent to a potential of the form in (\ref{eq:potregpoints}) by the gauge
    \begin{equation}\label{eq:gaugeupright}
      g_1=\begin{pmatrix} S^{1/2} & 0\\ 0 & S^{-1/2}\end{pmatrix}.
    \end{equation}

    \subsection{The $z^AP$ lemma}
      Suppose $\dd\Phi=\Phi\eta$ for which $\eta=A\frac{dz}{z-z_k}+\bigO(z^0)\;dz$ has a simple pole at $z=z_k$ and Delaunay residue $A$. A standard result in the theory of regular singularities states that under certain conditions on the eigenvalues of $A$, there exists a solution of the form $\Phi=z^AP=\exp\left(A\log{z}\right)P$, where $P$ extends holomorphically to $z=z_k$ (see \cite[Lemma 14]{KilKRS}). \Cref{zap} summarizes these ideas for our context.
      \begin{lemma}\label{zap}
        For $k\in\lbrace 0, 1\rbrace$, the \ODE $\;\dd\Phi=\Phi\xi$ has solutions
        \begin{equation}\label{eq:zapsolutions}
          \Phi_k (z)=\exp\left(\Delta_k\log\left((-1)^k(z-z_k)\right)\right)P_k(z)g_0(z),
        \end{equation}
        at $z_k$, where $P_k(z)$ is holomorphic at $z=z_k$ and $P_k(z_k)=H_k^{-1}$.
      \end{lemma}
      \begin{proof}
        The potential $\eta$ has a simple pole at $z_k$ with residue $A_k$. Since $\vartheta_k\notin\frac{1}{2}\Z\out 0$, by the theory of regular singular points, there exists a solution to the \ODE $\;d\Psi=\Psi\eta$ of the form
        \begin{equation*}
          \Psi_k(z)=\exp\left(A_k\log\left(z-z_k\right)\right)Q_k(z),
        \end{equation*}
        where $Q_k(z)$ is holomorphic at $z=z_k$ and $Q_k(z_k)=\identity$. Since $\xi=\eta.g_0$, then
        \begin{equation*}
          \hat{\Phi}_k(z)=H_k^{-1}\exp\left(A_k\log\left(z-z_k\right)\right)Q_k(z)g_0(z),
        \end{equation*}
        is a solution to the \ODE $\dd\Phi=\Phi\xi$. Since $A_kH_k=H_k\Delta_k$, then
        \begin{equation*}
          \hat{\Phi}_k(z)=\exp\left(\Delta_k\log\left(z-z_k\right)\right)H_k^{-1}Q_k(z)g_0(z).
        \end{equation*}
        The theorem follows with $P_k:=H_k^{-1}Q_k$, $\Phi_0:=\hat{\Phi}_0$ and $\Phi_1:=\diag\left((-1)^{-\vartheta_1},(-1)^{\vartheta_1}\right)\hat{\Phi}_1$.
      \end{proof}
      %
    %
    %%%%%%%%%%%%%%%%%%%%%%%%%%%%%%%%%%%%%%%%%%%%%%%%%%%%%%%%%%%%%%%%%%%%%%%%%
  %
  %%%%%%%%%%%%%%%%%%%%%%%%%%%%%%%%%%%%%%%%%%%%%%%%%%%%%%%%%%%%%%%%%

  \section{The Confluent Heun Equation}\label{sec:che}
    
    We have illustrated in sections \ref{sec:associatedode} and \ref{sec:regularpoints} why we want to prescribe two regular singularities and  one irregular singularity in our potential and, consequently, in the scalar \ODE (\ref{eq:scalar2ndode}) associated to the initial value problem (\ref{eq:IVP}). Let us do this now explicitly.

    The simplest second order \ODE with two regular singularities and one irregular singularity is the confluent Heun equation (CHE), which in its \emph{non-symmetrical canonical form} is written as
    \begin{equation}\label{eq:standardCHE}
      y''+\left(4p+\frac{\gamma}{z}+\frac{\delta}{z-1}\right)y'+\frac{4p\alpha z-\sigma}{z(z-1)}y=0
    \end{equation}
    where the parameters $p, \gamma, \delta, \alpha, \sigma\in\C$. This equation arises as a result of the confluence of two regular singular points in the Heun equation \cite{Heun}, and counts the points $z_0=0$ and $z_1=1$ as regular singularities and $z_\infty=\infty$ as an irregular singular point of rank 1. For the purpose of this work we are going to consider the CHE written as in \cite{ScSc}, that is
    \begin{align}\label{eq:CHE}
      \begin{split}
        y''&+\left(2a+\frac{1-\mu_0}{z}+\frac{1-\mu_1}{z-1}\right)y'\\
        &+\left(\frac{z[a(2-\mu_0-\mu_1)-(r_0+r_1)]+\frac{1}{2}(\mu_0\mu_1-2a(1-\mu_0)-(\mu_0+\mu_1)+2r_0+1)}{z(z-1)}\right)y=0,
      \end{split}
    \end{align}
    The parameters $\mu_0, \mu_1, a, r_0, r_1$ are again complex, but we will need to restrict our choices later on.
    
    Consider again the off-diagonal potential in (\ref{eq:offdpot}) with functions $\nu:=e^{2 f}R$ and $\rho:=e^{-2 f}T$, where $f=\int S(w)\;dw$, obtained in (\ref{eq:1pot}). Plugging $\nu$ and $\rho$ into \cref{eq:scalar2ndode}, it is easy to find expressions for the functions $R$, $S$ and $T$. In particular, one way to prescribe the CHE (\ref{eq:CHE}) in the potential is with the following relations:
    \begin{align}\label{eq:RST}
      R(z)&:=(z-1)^{\mu_1-1}z^{\mu_0-1},\nonumber\\
      S(z)&:=-a,\\
      T(z)&:=-z^{-\mu_0}(z-1)^{-\mu_1}[z(a(2-\mu_0-\mu_1)-(r_0+r_1))\nonumber\\
      &+\frac{1}{2}(\mu_0\mu_1-2a(1-\mu_0)-(\mu_0+\mu_1)+2r_0+1)].\nonumber
    \end{align}
    Using these functions in the potential (\ref{eq:1pot}) and doing the gauge considered in \cref{sec:associatedode} we end up with an off-diagonal potential which has associated (\ref{eq:CHE}) as scalar \ODE. Also, the correspondence with the parameters and functions used in the gauge in \cref{sec:regularpoints} is as follows: set $z_0=0$ and $z_1=1$ and for $k\in\lbrace 0,1\rbrace$ let
    \begin{align}\label{eq:abthetas}
      \vartheta_k&:=\frac{1-\mu_k}{2},\nonumber\\
      a(z)&:=a,\\
      b(z)&:=\frac{r_0}{z}+\frac{r_1}{z-1}+a^2.\nonumber
    \end{align}
    Sch\"afke and Schmidt give in \cite{ScSc} an asymptotic formula for the connection coefficients between a set of two solutions of \cref{eq:CHE} around $z=0$ and another set of two solutions of \cref{eq:CHE} around $z=1$, in terms of their series expansion coefficients. From the two equations appearing in Proposition 2.14 in \cite{ScSc}, we only need the first one in order to write down our matrix relationship $C=\Phi_0\Phi_1^{-1}$, as in what follows we will only consider the unique solution at $z=0$. These connection coefficients can be written in a matrix form allowing us to define the \emph{connection matrix} between two solutions $\Phi_0$ and $\Phi_1$. Let
    \begin{equation}\label{eq:seriessol}
      y_0(z)=\Gamma(1-\mu_0)\sum_{k=0}^\infty c_k z^k
    \end{equation}
    be the series expansion of the unique solution to \cref{eq:CHE} which is holomorphic at $z=0$ and satisfies $y_0(0)=1$.
    \begin{theorem}\cite{ScSc}\label{connection}
      The connection matrix $C:=\Phi_0\Phi_1^{-1}$ is
      \begin{equation}\label{eq:C}
        C=\begin{pmatrix} \frac{\Gamma(\mu_0)}{\Gamma(1-\mu_0)} & 0 \\ 0 & 1\end{pmatrix}\begin{pmatrix*}[r] q(\mu_0,-\mu_1) & q(\mu_0,\mu_1)\\ q(-\mu_0,-\mu_1) & q(-\mu_0,\mu_1) \end{pmatrix*}\begin{pmatrix} 1 & 0 \\ 0 & \frac{\Gamma(\mu_1)}{\Gamma(1-\mu_1)}\end{pmatrix},
      \end{equation}
      where the asymptotic formula by which $q$ can be calculated explicitly is
      \begin{equation}\label{eq:q}
        q(\mu_0,\mu_1)=\Gamma(1-\mu_0)\Gamma(1-\mu_1)\lim_{k\to\infty}\frac{\Gamma(k+1)}{\Gamma(k-\mu_1)}c_k.
      \end{equation}
    \end{theorem}
    \begin{proof}
      The proof amounts to converting the notation of Sch\"afke-Schmidt to our notation. Define
      \begin{equation}\label{eq:notationScSc}
        \tilde{\Phi}_k=G^{-1}D_k^{-1}\Phi_kg_0^{-1},\quad D_k:=\diag\left(\Gamma(\mu_k),(-1)^k\Gamma(1-\mu_k)\right),\quad k\in\lbrace 0,1\rbrace.
      \end{equation}
      By \cite[Proposition 2.14, Theorem 2.15]{ScSc}, the connection matrix $\tilde{C}=\tilde{\Phi}_0\tilde{\Phi}_1^{-1}\in\GL\C$ is
      \begin{equation}\label{eq:CScSc}
        \tilde{C}=\Gamma(\mu_1)\Gamma(1-\mu_1)\begin{pmatrix*}[r] \tilde{q}(\mu_0,-\mu_1) & -\tilde{q}(\mu_0,\mu_1)\\ \tilde{q}(-\mu_0,-\mu_1) & -\tilde{q}(-\mu_0,\mu_1) \end{pmatrix*},
      \end{equation}
      where
      \begin{equation}\label{eq:qScSc}
        \tilde{q}(\mu_0,\mu_1)=\lim_{k\to\infty}\frac{\Gamma(k+1)}{\Gamma(k-\mu_1)}c_k.
      \end{equation}
      The theorem follows by the relations between our notation and that of \cite{ScSc}:
      \begin{equation}\label{eq:relationsnotation}
        q(\mu_0,\mu_1)=\Gamma(1-\mu_0)\Gamma(1-\mu_1)\tilde{q}(\mu_0,\mu_1)\quad\text{and}\quad C=D_0\tilde{C}D_1^{-1}.
      \end{equation}
    \end{proof}
    By standard methods (see \cite[Part B, 2.2]{Heun}), the sequence $\lbrace c_k\rbrace$ of coefficients defined by \cref{eq:seriessol} satisfy the 3-term recurrence
    \begin{equation}\label{eq:recurrence}
      U(k)c_{k+1}=V(k)c_k+W(k)c_{k-1},\quad c_{-1}=0,\; c_0=1
    \end{equation}
    where
    \begin{align}\label{eq:polynomials}
      U(k)&:=(1+k)(1+k-\mu_0),\nonumber\\
      V(k)&:=k(k+1-2a-\mu_0-\mu_1)+\frac{1}{2}(\mu_0-1)(\mu_1-1)+a(\mu_0-1)+r_0,\\
      W(k)&:=a(2k-\mu_0-\mu_1)-(r_0+r_1).\nonumber
    \end{align}
    The recurrence in (\ref{eq:recurrence}) and its polynomials in (\ref{eq:polynomials}) will be used to obtain unitarisability in \cref{sec:unitarisablitymonodromy}.
  %
  %%%%%%%%%%%%%%%%%%%%%%%%%%%%%%%%%%%%%%%%%%%%%%%%%%%%%%%%%%%%%%%%%

  \section{Unitarisability of the Monodromy}\label{sec:unitarisablitymonodromy}

    The proof of the next proposition is deferred to the appendix.
    \begin{proposition}\label{unit2mat}
      Let $M_0,M_1\in\SL(\C)\out{\pm\identity}$ be irreducible and individually unitarisable. Let $\varphi, \varphi'\in\C P^1$ and $\psi, \psi'\in\C P^1$ be the respective eigenlines of $M_0$ and $M_1$. Then $M_0$ and $M_1$ are simultaneously unitarisable if and only if the cross-ratio 
      \begin{equation}\label{eq:negcrossratio}
        \left[\varphi,\psi,\varphi',\psi'\right]\in\R_-.
      \end{equation}
    \end{proposition}
    In our context, the two matrices to be unitarised are of the form $\Delta_0$ and $C\Delta_1C^{-1}$, where $\Delta_0$ and $\Delta_1$ are diagonal. The unitarisability criterion in \Cref{unit2mat} for this case can be expressed as follows.
    \begin{proposition}\label{unitM0M1}
      Consider two matrices $M_0:=\Delta_0$ and $M_1:=C\Delta_1C^{-1}$, where $\Delta_0,\Delta_1\in\SL(\C)\out{\pm\identity}$ are diagonal matrices and
      \begin{equation}\label{eq:Cunit}
        C=:\begin{pmatrix} a & b\\ c & d \end{pmatrix}\in\SL(\C).
      \end{equation}
      \begin{enumerate}[label=(\roman*)]
        \item\label{i} $M_0$ and $M_1$ are irreducible if and only if $a, b, c$ and $d$ are non-zero.
        \item\label{ii} If $M_0$ and $M_1$ are irreducible and individually unitarisable, then $M_0$ and $M_1$ are simultaneously unitarisable if and only if the ratio $\frac{bc}{ad}\in\R_{-}$.
      \end{enumerate}
    \end{proposition}
    \begin{proof}
      For \ref{i}, since $M_0$ is diagonal and not $\pm\identity$, then $M_0$ and $M_1$ are reducible if and only if $M_1$ is upper or lower triangular. Writing $\Delta_1=\diag(\beta,\beta^{-1})$, compute
      \begin{equation}\label{eq:M1}
        M_1=C\Delta_1C^{-1}=\begin{pmatrix} ad\beta-bc\beta^{-1} & -ab(\beta-\beta^{-1})\\ cd(\beta-\beta^{-1}) & ad\beta^{-1}-bc\beta \end{pmatrix}.
      \end{equation}
      Then $M_1$ is upper or lower triangular if and only if at least one of $a,b,c$ or $d$ vanishes.

      To prove \ref{ii}, since $M_0$ is diagonal, its eigenlines in $\C P^1$ are $\varphi_1=0$ and $\varphi_2=\infty$. Since $M_1C=C\Delta_1$, the eigenlines of $M_1$ in $\C^2$ are the columns of $C$, so its eigenlines in $\C P^1$ are $\psi_1=a/c$ and $\psi_2=b/d$. Then
      \begin{equation}\label{eq:creigenlines}
        [\varphi_1,\psi_1,\varphi_2,\psi_2]=\frac{\psi_2}{\psi_1}=\frac{bc}{ad}.
      \end{equation}
      By \Cref{unit2mat}, $M_0$ and $M_1$ are simultaneously unitarisable if and only if this cross ratio is in $\R_{-}$.
    \end{proof}
    We will apply the above criterion to the monodromy of an \ODE as follows. Consider a potential $\xi$ with  singular points $z_0$ and $z_1$. Let $\tau_0$ and $\tau_1$ be the deck transformations corresponding to closed paths around $z_0$ and $z_1$ respectively. Let $\Phi_0$ and $\Phi_1$ be solutions to the \ODE $\;\dd\Phi=\Phi\xi$ chosen so that their respective monodromies
    \begin{equation*}
      \Delta_0:=\Phi_0(\tau_0(z))\Phi_0^{-1}\quad\text{and}\quad\Delta_1:=\Phi_1(\tau_1(z))\Phi_1^{-1}
    \end{equation*}
    are diagonal. Let $C=\Phi_0\Phi_1^{-1}$ be the connection matrix between these two solutions with entries as in (\ref{eq:Cunit}). The monodromies of $\Phi_0$ at $z_0$ and $z_1$ are respectively
    \begin{equation*}
      M_0=\Delta_0\quad\text{and}\quad M_1=\Phi_0(\tau_1(z))\Phi_0^{-1}=C\Phi_1(\tau_1(z))\Phi_1^{-1}C^{-1}=C\Delta_1C^{-1}.
    \end{equation*}
    Suppose $M_0$ and $M_1$ are irreducible and individually unitarisable. By \Cref{unitM0M1}, $M_0$ and $M_1$ are simultaneously unitarisable if and only if $\frac{bc}{ad}\in\R_{-}$.
    For the remainder of the paper we choose all parameters in the CHE to be real. Also, we assume $a>0$. Is easy to check that the coefficients $U(k)$, $V(k)$ and $W(k)$ of the CHE recurrence (\ref{eq:recurrence}) are positive for all sufficient large $k$. Under this assumption, the signs of the entries of the connection matrix $C$ in (\ref{eq:C}) can be computed as follows.
    \begin{proposition}\label{signsrecursion}
      Suppose $\mu_j<1$ for $j\in\lbrace 0,1\rbrace$ and that there exists $k_0\in\N$ such that $U(k)>0$, $V(k)>0$ and $W(k)>0$ for all $k\geq k_0$. If $c_{k_0-1}>0$ and $c_{k_0}>0$, then $q$ defined in \cref{eq:q} satisfies $q\geq 0$. Similarly, if $c_{k_0-1}<0$ and $c_{k_0}<0$, then $q\leq 0$.
    \end{proposition}
    \begin{proof}
      Since $\mu_j<1$ for $j\in\lbrace 0,1\rbrace$ and the function $\Gamma(x)>0$ for all $x>0$, we have that $\Gamma(1-\mu_j)>0$ for $j\in\lbrace 0,1\rbrace$, $\Gamma(k+1)>0$ and $\Gamma(k-\mu_1)>0$ for all $k\geq k_0$.

      By hypothesis $U(k)$, $V(k)$ and $W(k)$ are positive for all $k\geq k_0$. Hence, in the case of $c_{k_0-1}>0$ and $c_{k_0}>0$, the third term $c_{k_0+1}$ must be positive as well. Then, by induction, $\lbrace c_k\rbrace_{k=k_0-1}^\infty$ are all positive coefficients. This implies that $q\geq 0$.

      Similarly, if $c_{k_0-1}<0$ and $c_{k_0}<0$ the coefficients $\lbrace c_k\rbrace_{k=k_0-1}^\infty$ are negative and therefore $q\leq 0$.
    \end{proof}
    The criterion for unitarisability in \Cref{unitM0M1}, the asymptotic formula for the connection matrix in \Cref{connection}, and the recurrence relation for the CHE solution in \cref{eq:recurrence} yield the following sufficient condition for the unitarisability of the monodromy. Write the parameters for the CHE as a 5-tuple $\chi:=(\mu_0,\mu_1,r_0,r_1,a)\in\R^5$. Define the finite integer
    \begin{equation}\label{eq:m}
      m(\chi):=\min_{k_0\in\N}\lbrace U(k,\chi)>0, V(k,\chi)>0\;\text{and}\;W(k,\chi)>0\;\text{for all}\;k\geq k_0\rbrace,
    \end{equation}
    and the sets
    \begin{align}\label{eq:SplusSminus}
      \mathcal{S}_+:=\lbrace \chi\in\R^5\mid\text{there exists}\;\ell\geq m(\chi)\;\text{such that}\;c_{\ell-1}>0\;\text{and}\;c_{\ell}>0\rbrace,\\
      \mathcal{S}_-:=\lbrace \chi\in\R^5\mid\text{there exists}\;\ell\geq m(\chi)\;\text{such that}\;c_{\ell-1}<0\;\text{and}\;c_{\ell}<0\rbrace.\nonumber
    \end{align}
    \begin{proposition}\label{5tuple}
      If each of the four 5-tuples $(\pm\mu_0,\pm\mu_1,r_0,r_1,a)\in\R^5$ lies in $\mathcal{S}_+\cup\mathcal{S}_-$, then the monodromy is unitarisable if and only if an odd number of these tuples lie in $\mathcal{S}_+$.
    \end{proposition}
    \begin{proof}
      By \Cref{signsrecursion}, if $\chi\in\mathcal{S}_+$ then $q(\chi)\geq 0$, and if $\chi\in\mathcal{S}_-$ then $q(\chi)\leq 0$. An odd number of the tuples lie in $\mathcal{S}_+$ if and only if
      \begin{equation}\label{eq:ratioqs}
        \frac{q(\mu_0,\mu_1)q(-\mu_0,-\mu_1)}{q(\mu_0,-\mu_1)q(-\mu_0,\mu_1)}\in\R_{-},
      \end{equation}
      that is, by the remarks after \Cref{unitM0M1}, if and only if the monodromy is unitarisable.
    \end{proof}
    %    
  %
  %%%%%%%%%%%%%%%%%%%%%%%%%%%%%%%%%%%%%%%%%%%%%%%%%%%%%%%%%%%%%%%%%

  \section{Construction of New Trinoids}\label{sec:trinoids}

    \subsection{The trinoid potential}\label{sec:potential}
      To construct trinoids we choose the potential
      \begin{equation}\label{eq:trinoidpotential}
        \xi_T=\begin{pmatrix} 0 & \lambda^{-1}\\ \lambda Q_T & 0\end{pmatrix} dz,
      \end{equation}
      where
      \begin{equation}\label{eq:QT}
        Q_T:=t\left(-\frac{w_0}{4z^2}-\frac{w_1}{4(z-1)^2}+\frac{\hat{r}_0}{z}+\frac{\hat{r}_1}{z-1}+p^2\right),\quad t:=-\frac{1}{4}\lambda^{-1}(\lambda-1)^2,
      \end{equation}
      and $w_0, w_1, \hat{r}_0, \hat{r}_1, p\in\R$ are free parameters. The parameters $w_0$ and $w_1$ will be the asymptotic end weights of the Delaunay ends at $0$ and $1$. The parameters $\hat{r}_0$ and $\hat{r}_1$ affect the weight of the irregular end, and $p$ the shape of the trinoid. Note that, defining $\lambda=e^{i\theta}$ with $\theta\in [0,2\pi]$, the value of $t=-\frac{1}{4}\lambda^{-1}(\lambda-1)^2=\sin^2\frac{\theta}{2}\in [0,1]$.
      
      With $\Lambda:=\diag(\lambda^{1/2},\lambda^{-1/2})$, the gauged potential 
      \begin{equation}\label{eq:gaugedpotential}
        \xi_T.(\Lambda^{-1})=\begin{pmatrix} 0 & 1\\ Q_T & 0\end{pmatrix} dz
      \end{equation}
      has the form of the CHE potential defined in sections \ref{sec:regularpoints} and \ref{sec:che}, where the coefficients $(\mu_0,\mu_1,r_0,r_1,a)$ in the CHE equation are related to the parameters $(w_0, w_1, \hat{r}_0, \hat{r}_1, p)$ in the trinoid potential $\xi_T$ by
      \begin{equation}\label{eq:parametersrelations}
        \mu_k=\sqrt{1-w_kt},\quad r_k=\hat{r}_kt,\quad a^2=p^2t,\quad k\in\lbrace 0,1\rbrace.
      \end{equation}
      The monodromy of the trinoid potential is unitarisable along $\Sp^1$ if and only if that of the gauged potential $\xi_T.(\Lambda^{-1})$ is.
    %
    %%%%%%%%%%%%%%%%%%%%%%%%%%%%%%%%%%%%%%%%%%%%%%%%%%%%%%%%%%%%%%%%%

      \subsection{Construction of trinoids}\label{sec:constructiontrinoids}

        \begin{theorem}\label{constructtrinoids}
          Let $\xi_T$ be a trinoid potential with unitarisable monodromy on $\Sp^1$ minus a finite set. Let $\Phi$ be a solution of $\;\dd\Phi=\Phi\xi_T$. Then there exists a positive dressing $h$ such that the \CMC immersion induced by $h\Phi$ via the generalized Weierstrass representation on the universal cover descends to the three-punctured sphere. The ends at $z=0$ and $z=1$ are asymptotic to Delaunay surfaces.
        \end{theorem}
        \begin{proof}
          By \cite{KilKRS}, there exists  a positive loop $h:\mathcal{D}_1\to\GL\C$ such that the monodromy of $h\Phi$ is unitary. The local unitary monodromies $M_0$ and $M_1$ satisfy the closing conditions $M_k(0)=\pm\identity$ and $M_k'(0)=0$, for $k\in\lbrace 0,1\rbrace$. Hence $h\Phi$ induces an immersion of the three-punctured sphere via the $\GL\C$ version of the generalised Weierstrass representation. The ends at $z=0$ and $z=1$ are asymptotic to Delaunay cylinders with respective weights $w_0$ and $w_1$ by \cite{KilRS}.
        \end{proof}
        \begin{remark}
          Due to the structure of $\xi_T$, any trinoid $T$ constructed from $\xi_T$ in fact lies in a one-parameter family of trinoids $T_\kappa$ with monotonically varying Delaunay end weights. If $(\mu_0,\mu_1,r_0,r_1,a)$ are the parameters for $T$, then the parameters for the family of trinoids $T_\kappa$ are $(\kappa w_0,\kappa  w_1, \kappa \hat{r}_0, \kappa \hat{r}_1, \sqrt{\kappa}p)$ with $\kappa$ ranging over the interval $(0,1]$.
        \end{remark}
        %
      %
      %%%%%%%%%%%%%%%%%%%%%%%%%%%%%%%%%%%%%%%%%%%%%%%%%%%%%%%%%%%%%%%%%

    \subsection{Unitarisability of the monodromy}\label{sec:unitarisableparameters}

      It remains to find values of the 5 parameters in $\xi_T$ so that the monodromy is unitarisable. An algorithm to test the hypotheses of \Cref{5tuple} is as follows. For a 5-tuple $\chi=(\mu_0,\mu_1,r_0,r_1,a)$, consider the $(k+1)$-coefficient of the recurrence in (\ref{eq:recurrence}), which is given by
      \begin{equation}\label{eq:ckplus1}
        c_{k+1}(\chi)=\frac{V(k,\chi)c_k(\chi)+W(k,\chi)c_{k-1}(\chi)}{U(k,\chi)}.
      \end{equation}
      The radicals appearing in $c_{\ell+1}(\chi)$ can be eliminated, reducing the problem to showing that a polynomial is positive in an interval. Let $\Theta=(w_0,w_1,\hat{r}_0,\hat{r}_1,p)\in\R^5$ be a choice of parameters for $\xi_T$. Note that $c_{\ell+1}(y_0,y_1,\hat{r}_0 x^2, \hat{r}_1 x^2, px)$ defines a rational function
      \begin{equation*}
        \frac{\mathcal{P}_\ell(y_0,y_1,\hat{r}_0 x^2, \hat{r}_1 x^2, px)}{\mathcal{Q}_\ell(y_0,y_1,\hat{r}_0 x^2, \hat{r}_1 x^2, px)}
      \end{equation*}
      for some polynomials $\mathcal{P}_\ell,\mathcal{Q}_\ell\in\R[x,y_0,y_1]$ depending on $\ell$ and $\Theta$. Define the polynomial functions $F_k,G_k\in\R[x,y_0,y_1]$
      \begin{align}\label{eq:FandG}
        F_k(x,y_0,y_1)&:=\mathcal{P}_k(y_0,y_1,\hat{r}_0 x^2, \hat{r}_1 x^2, px),\\
        G_k(x,y_0,y_1)&:=F_k(x,y_0,y_1)F_k(x,y_0,-y_1)F_k(x,-y_0,y_1)F_k(x,-y_0,-y_1).\nonumber
      \end{align}
      Since $G_k$ is even in $y_0$ and in $y_1$, then the function $f_k$ depending on $k$ and $\Theta$
      \begin{equation}\label{eq:littlef}
        f_k(x):=G_k\left(x,\sqrt{1-w_0 x^2},\sqrt{1-w_1 x^2}\right)
      \end{equation}
      is in $\R[x]$.
      \begin{proposition}\label{signsparameters}
        Let $\Theta:=(w_0,w_1,\hat{r}_0,\hat{r}_1,p)\in\R^5$ be a choice of parameters for the trinoid potential $\xi_T$ and let
        \begin{equation}\label{eq:signedparametersrelations}
          \chi_{\pm\pm}:=(\pm\mu_0,\pm\mu_1,r_0,r_1,a)=\left(\pm\sqrt{1-w_0t},\pm\sqrt{1-w_1t},\hat{r}_0t,\hat{r}_1t,p\sqrt{t}\right).
        \end{equation}
        Let $k_0\in\N$ be such that for each of the four choices of signs, $U(k,\chi_{\pm\pm})>0$, $V(k,\chi_{\pm\pm})>0$ and $W(k,\chi_{\pm\pm})>0$ for all $k\geq k_0$. Suppose
        \begin{enumerate}[label=(\roman*)]
          \item\label{iparam} $f_{k_0-1}(x)\neq 0$ and $f_{k_0}(x)\neq 0$ along $x\in(0,1)$,
          \item\label{iiparam} for each of the four choices $\chi_{\pm\pm}$, and some $t_0\in(0,1)$,
            \begin{equation}\label{eq:coefficientssigns}
              \sign c_{k_0-1}(\chi_{\pm\pm}(t_0))=\sign c_{k_0}(\chi_{\pm\pm}(t_0)),
            \end{equation}
          \item\label{iiiparam} of the four signs in \ref{iiparam} and odd number are $+$ and an odd number are $-$.
        \end{enumerate}
        Then, the monodromy with parameters $\Theta$ is unitarisable.
      \end{proposition}
      \begin{proof}
        By its definition, $f_{\ell+1}(x)$ has a zero along $x\in(0,1)$ if and only if at least one of the four functions $c_{\ell+1}(\chi_{\pm\pm}(t))$ has a zero along $t\in(0,1)$. Thus by \ref{iparam}, none of the eight functions $c_{k_0-1}(\chi_{\pm\pm}(t))$ and $c_{k_0}(\chi_{\pm\pm}(t))$ has a zero along $t\in(0,1)$. By \ref{iiparam} and continuity, all $\chi_{\pm\pm}\in\mathcal{S}_+\cup\mathcal{S}_-$, where $\mathcal{S}_+$ and $\mathcal{S}_-$ are the sets defined in (\ref{eq:SplusSminus}). The monodromy is unitarisable by \Cref{signsparameters}\ref{iiiparam} and \Cref{5tuple}.
      \end{proof}
      \begin{figure}[t] 
        \centering
      	\includegraphics[width=5cm]{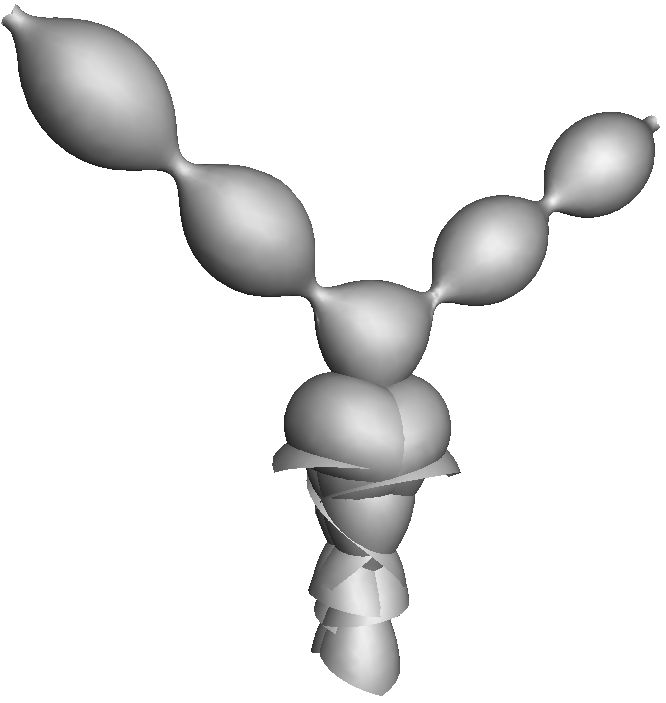}
      	\includegraphics[width=5.5cm]{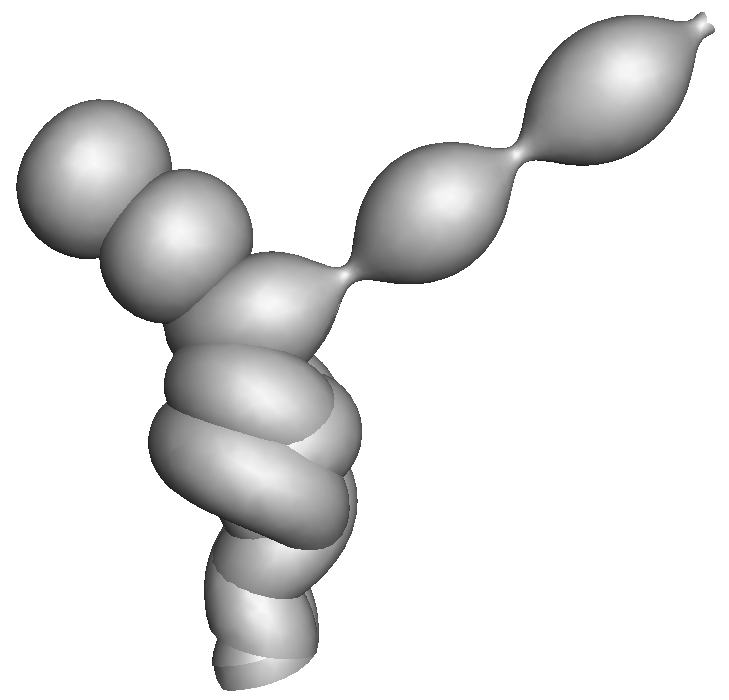}
      	\includegraphics[width=4.1cm]{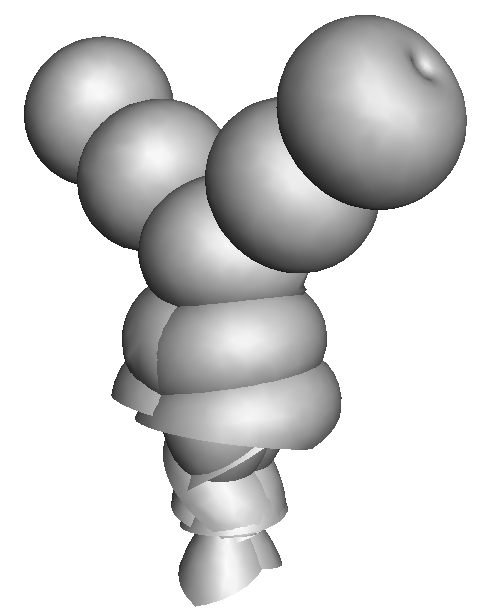}    
        \caption{ \label{fig:examples} \footnotesize Trinoids with one irregular end and two Delaunay ends.  Graphics were produced with CMCLab~\cite{Sch}. The regular end weights are either $\tfrac{1}{2}$ or $-\tfrac{1}{2}$ while the irregular end weights vary. The parameters $(w_0,w_1,\hat{r}_0,\hat{r}_1,p)$ used to construct each of them are $\left(\frac{1}{2},\frac{1}{2},-\frac{1}{8},\frac{1}{8},\frac{1}{8}\right),\quad\left(\frac{1}{2},-\frac{1}{2},-\frac{1}{8},\frac{1}{4},\frac{1}{8}\right),\quad\text{and}\quad\left(-\frac{1}{2},-\frac{1}{2},-\frac{1}{80000000},0,\frac{1}{8}\right)$.}
      \end{figure}
      The examples in \cref{fig:examples} were computed by \Cref{constructtrinoids} and the unitarisability criterion in \Cref{signsparameters}.
      \begin{proposition}\label{5parameters}
        The conditions of \Cref{signsparameters} are satisfied for the choice of parameters
        \begin{equation}\label{eq:5param}
          \Theta=\left(\frac{1}{2},\frac{1}{2},-\frac{1}{8},\frac{1}{8},\frac{1}{8}\right).
        \end{equation}
      \end{proposition}
      \begin{proof}
        For this choice of $\Theta$, we compute
        \begin{align}\label{eq:polynomialfunctions}
          f_1(x)&=x^4\left(393216 - 245760 x + 243712 x^2 - 57856 x^3 +\right.\nonumber \\
          &\hspace{1cm} \left. 26368 x^4 - 5120 x^5 + 1008 x^6 - 252 x^7 + 27 x^8\right)\nonumber \\
          &\times\left(-1179648 + 1032192 x - 2260992 x^2 - 61440 x^3 +\right.\nonumber \\
          &\hspace{1cm} \left. 1267968 x^4 - 241088 x^5 + 189328 x^6 - 41116 x^7 + 6859 x^8\right)\nonumber \\
          f_2(x)&=x^4\left(21743271936 - 6794772480 x + 6455033856 x^2 - 2021916672 x^3 +\right.\nonumber \\
          &\hspace{1cm} \left. 762642432 x^4 - 174735360 x^5 + 41717760 x^6 -\right. \\
          &\hspace{1cm} \left. 5871616 x^7 + 1031424 x^8 - 109248 x^9 + 12240 x^{10} - 1512 x^{11} + 81 x^{12}\right) \nonumber \\
          &\times\left(-79725330432 + 37144756224 x - 24310185984 x^2 +\right.\nonumber \\
          &\hspace{1cm} \left. 358612992 x^3 + 24553881600 x^4 - 9189408768 x^5 +\right.\nonumber \\
          &\hspace{1cm} \left. 7316135936 x^6 - 1422535168 x^7 + 570771456 x^8 -\right.\nonumber \\
          &\hspace{1cm} \left. 78774080 x^9 + 16092368 x^{10} - 1562408 x^{11} + 130321 x^{12}\right) \nonumber
        \end{align}
        Sturm's theorem applied to $f_1$ and $f_2$ shows that these two polynomials have no zero on the interval $(0,1]$. This verifies \Cref{signsparameters}$(i)$. The conditions $(ii)$ and $(iii)$ are verified by computing at $t_0=\frac{4}{5}$, for $k\in\lbrace 1,2\rbrace$, yielding $c_{k}(\chi_{++}(t_0))<0$, $c_{k}(\chi_{+-}(t_0))>0$, $c_{k}(\chi_{-+}(t_0))>0$ and $c_{k}(\chi_{--}(t_0))>0$.
      \end{proof}
      \begin{theorem}\label{5parameterfamily}
        There exists a five parameter family of \CMC trinoids with two Delaunay ends and one irregular end.
      \end{theorem}
      \begin{proof}
        The coefficients $U(k)$, $V(k)$ and $W(k)$ of the recurrence in (\ref{eq:recurrence}) depend holomorphically on the parameters of the choice $\Theta$. Each term of the sequence $c_{k+1}(\chi)$ is a rational function $\frac{\mathcal{P}_k(\chi)}{\mathcal{Q}_k(\chi)}$, where $\mathcal{P}_k$ and $\mathcal{Q}_k$ are polynomials in $W(0),\dots,W(k)$ and in $V(0),\dots,V(k)$ and $U(0),\dots,U(k)$ respectively. Thus they also depend holomorphically on the parameters. Note that under the assumptions made for the CHE parameters, in particular $\mu_0<1$, the polynomial $U(k)$ is never zero. Therefore, the function $f_k$ defined in (\ref{eq:littlef}) also depends holomorphically on the parameters.
        Let $\Theta$ be as in \Cref{5parameters}, for which we have checked that the conditions of \Cref{signsparameters} are satisfied. The polynomials $f_1$ and $f_2$ have a zero of order $4$ at $x=0$. A calculation shows that this order is preserved under a small perturbation of $\Theta\in\R^5$. For $i\in\lbrace 1,2\rbrace$, let us denote $f:=f_i$. We have that $f(x)=x^4g(x)$, where $g$ has no zeros on $[0,1]$, that is, $\varepsilon:=\inf_{x\in[0,1]}|g(x)|>0$. We can make a small perturbation of $g$ by choosing $\tilde{g}$ such that $|g-\tilde{g}|_{[0,1]}<\varepsilon$. If we consider $\tilde{f}(x):=x^4\tilde{g}(x)$, then we obtain that
        \begin{align}\label{eq:perturbations1}
          |\tilde{f}(x)|&=x^4|\tilde{g}(x)|=x^4|g(x)-(g(x)-\tilde{g}(x))|\nonumber\\
          &\geq x^4\left||g(x)|-|g(x)-\tilde{g}(x)|\right|\\
          &>x^4|\varepsilon-\varepsilon|=0.\nonumber
        \end{align}
        Hence, $\tilde{f}$ has no zeros on $(0,1]$. It is also easy to check that $x=0$ is not a zero of the perturbation $\tilde{g}$:
        \begin{equation}\label{eq:perturbations2}
          |\tilde{g}(0)|=|g(0)-(g(0)-\tilde g(0))|\geq ||g(0)|-|g(0)-\tilde g(0)||>|\varepsilon-\varepsilon|=0.
        \end{equation}
        It follows that the condition in \Cref{signsparameters}$(i)$ is preserved under small perturbations of $\Theta$. Also conditions $(ii)$ and $(iii)$ are trivially preserved under such perturbations. Hence by \Cref{signsparameters}, the monodromy of $\xi_T$ is unitarisable in a small neighborhood of $\Theta\in\R^5$. \Cref{constructtrinoids} constructs a five parameter family of trinoids with one irregular end.
      \end{proof}
      %
    %
    %%%%%%%%%%%%%%%%%%%%%%%%%%%%%%%%%%%%%%%%%%%%%%%%%%%%%%%%%%%%%%%%%

    \subsection{End weights}\label{sec:endweights}

      We conclude by computing the weight at the irregular end of a trinoid, in a similar fashion as in \cite{KilSch}. Let
      \begin{equation}\label{eq:potentialweights1}
        \xi=\begin{pmatrix} 0 & 1\\ tQ & 0\end{pmatrix} dz
      \end{equation}
      be a potential where $Q$ is holomorphic on the circle $|z|=2$ with Laurent series $Q=\sum_{k=-\infty}^\infty a_kz^k$. The first few terms of the series of the monodromy $M$ of the solution of $\;\dd\Phi=\Phi\xi$, $\Phi(2)=\identity$ along the circle $|z|=2$ is as follows.
      \begin{proposition}\label{seriesmonodromy}
        Let $M$ be the monodromy with respect to the curve $\gamma(s) = 2e^{is}$, $s\in[0,2\pi]$ of $\Phi$ for the trinoid potential $\xi_T$ with $\Phi(2) = \identity$. Define $\lambda=e^{i\theta}$ and let $\sum_{k=0}^\infty M_k\theta^k$ be the series expansion of $M$ along $|z|=2$. Then $M_0=\identity$, $M_1=0$ and
        \begin{equation}\label{eq:M2}
          M_2=2\pi i\begin{pmatrix} \frac{a_{-2}}{4}-\frac{a_{-1}}{2} & a_{-2}-a_{-1}-\frac{a_{-3}}{4}\\ \frac{a_{-1}}{4} & \frac{a_{-1}}{2}-\frac{a_{-2}}{4}\end{pmatrix}.
        \end{equation}
      \end{proposition}
      \begin{proof}
        By the gauge $\Lambda:=\diag(\lambda^{1/2},\lambda^{-1/2})$ we obtain
        \begin{equation}\label{eq:potentialweights2}
          \xi=\xi_T.(\Lambda^{-1})=\xi_0+t\xi_1,\quad\xi_0=\begin{pmatrix} 0 & \alpha\\ 0 & 0\end{pmatrix},\quad\xi_1=\beta\begin{pmatrix} 0 & 0\\ 1 & 0\end{pmatrix},
        \end{equation}
        where
        \begin{equation}\label{eq:alphabetat}
          \alpha=dz,\quad\beta=Q\;dz\;\text{ and }\; t=-\frac{1}{4}\lambda^{-1}(\lambda-1)^2=\sin^2\frac{\theta}{2}.
        \end{equation}
        Let $\Psi=\Psi_0+\Psi_1t+\bigO(t^2)$ be the solution to the initial value problem
        \begin{equation}\label{eq:IVPpsi}
          d\Psi=\Psi\xi,\quad\Psi(2)=\identity,
        \end{equation}
        and let $P=P_0+P_1t+\bigO(t^2)$ be the monodromy of $\Psi$. To compute $P_0$ and $P_1$, equate the coefficients as powers of $t$ in
        \begin{equation}\label{eq:powerst}
          d\Psi_0+d\Psi_1t+\bigO(t^2)=\left(\Psi_0+\Psi_1t+\bigO(t^2)\right)\left(\xi_0+\xi_1t\right)
        \end{equation}
        to obtain the two equations
        \begin{equation}\label{eq:IVPpsi0}
          d\Psi_0=\Psi_0\xi_0,\quad\Psi_0(2)=\identity,
        \end{equation}
        \begin{equation}\label{eq:IVPpsi1}
          d\left(\Psi_1\Psi_0^{-1}\right)=\Psi_0\xi_1\Psi_0^{-1},\quad\Psi_1(2)=0.
        \end{equation}
        The solution to (\ref{eq:IVPpsi0}) is
        \begin{equation}\label{eq:solpsi0}
          \Psi_0=\begin{pmatrix} 1 & \int\alpha\\ 0 & 1\end{pmatrix},
        \end{equation}
        where the path integral is along a path based at $2$. Since $\int\alpha=z-2$, then $\int_\gamma\alpha=0$ for $\gamma(s)=2e^{is}$, $s\in[0,2\pi]$. Hence $P_0=\identity$. Solve (\ref{eq:IVPpsi1}) by computing
        \begin{equation}\label{eq:solpsi1}
          \Psi_1\Psi_0^{-1}=\int\Psi_0\xi_1\Psi_0^{-1}=\int\beta\begin{pmatrix} \int\alpha & -\left(\int\alpha\right)^2\\ 1 & -\int\alpha\end{pmatrix}=\int\beta\begin{pmatrix} z-2 & -\left(z-2\right)^2\\ 1 & 2-z\end{pmatrix}.
        \end{equation}
        We are integrating along a path enclosing the two singular points so, by the residue theorem, we obtain
        \begin{equation}\label{eq:p1}
          P_1=\left(\int_\gamma\left(\Psi_0\xi_1\Psi_0^{-1}\right)\right)P_0=2\pi i\begin{pmatrix} a_{-2}-2a_{-1} & 4a_{-2}-4a_{-1}-a_{-3}\\ a_{-1} & 2a_{-1}-a_{-2}\end{pmatrix}.
        \end{equation}
        The series for the monodromy $M$ of $\Phi$ now follows from $M = \Lambda^{-1} P \Lambda$ and expressing the series in terms of $\theta$ as defined above.
      \end{proof}
      The \emph{force} associated to an element in the fundamental group \cite{KKS, Bob, KGB} is the matrix $A\in\su$ in the series expansion of the monodromy
      \begin{equation}\label{eq:seriesexpmonodromy}
        M=\identity+A\theta^2+\bigO(\theta^3)
      \end{equation}
      where $\lambda = e^{i\theta}$. The force is a homomorphism from the fundamental group to $\su\cong\R^3$. Its length $|A|=\sqrt{\det A}$ is the weight of the end.
      \begin{proposition}\label{weights}
        The weights of a trinoid constructed from $\xi_T$ with parameters $(w_0,w_1,\hat{r}_0,\hat{r}_1,p)$ at $z=0,1,\infty$ are respectively $w_0,w_1,w_\infty$ where
        \begin{equation}\label{eq:infinityweight}
          w_\infty=\frac{\pi}{8}\sqrt{(w_0+w_1)^2+8\hat{r}_0(w_1-2\hat{r}_1)-8\hat{r}_1w_0}.
        \end{equation}
      \end{proposition}
      \begin{proof}
        By \Cref{seriesmonodromy}, the weight of the irregular end of a trinoid constructed using its potential (\ref{eq:trinoidpotential}) is given by
        \begin{equation}\label{eq:sqrtdetA}
          \frac{\pi}{2}\sqrt{a_{-2}^2-a_{-1}a_{-3}}
        \end{equation}
        where $a_k$ are the Laurent coefficients of $Q$ as before. The result follows by a computation of these coefficients.
      \end{proof}
      \begin{remark}
        The three weights (lengths of the weight vectors) determine the three weight vectors. It is unknown to the authors if a notion of end axis can be established for an irregular end. In the case of all three ends being regular, a necessary condition for the unitarisability of the monodromy comes from the balancing formula \cite{KGB, KilKRS}: If the monodromy is unitary, the weight vectors $W_0, W_1, W_\infty\in\R^3$ satisfy $W_0 + W_1 + W_\infty = 0$. The weights are $w_k=\pm|W_k|$. It follows that $|w_i|\leq|w_j|+|w_k|$ for all permutations of $(0,1,\infty)$. A counterexample can be found in the presence of one irrgular end: For the parameters $(w_0,w_1,\hat{r}_0,\hat{r}_1,p)=\left(\frac{1}{2},\frac{1}{4},\frac{1}{4},\frac{17}{128},\frac{1}{8}\right)$ the resulting monodromy is unitarisable, but the balancing formula does not hold for all permutations of $(0,1,\infty)$. This counterexample questions the suitability of \cref{eq:infinityweight} as the way of measuring the end weight at $\infty$. Since with this definition the notion of balancing is not preserved, it might be interesting if this could be modified so that balancing still holds.
      \end{remark}
      %
    %
    %%%%%%%%%%%%%%%%%%%%%%%%%%%%%%%%%%%%%%%%%%%%%%%%%%%%%%%%%%%%%%%%%%%%%%%%%
  %
  %%%%%%%%%%%%%%%%%%%%%%%%%%%%%%%%%%%%%%%%%%%%%%%%%%%%%%%%%%%%%%%%%

  \appendix
  \section{Appendix: Geometry of Unitarisability}\label{sec:unitarisablity}

    \renewcommand{\theequation}{A-\arabic{equation}}  % redefine the command that creates the equation no.
    \setcounter{equation}{0}  % reset counter
    
    It remains to prove \Cref{unit2mat}, which give a criterion for the simultaneous unitarisability of two matrices in $\SL(\C)$ in terms of their eigenlines.

    \subsection{Hyperbolic $3$-space}\label{sec:hyperbolic}
      Hyperbolic $3$-space $\Hyp^3$ can be identified with the quotient $\SL(\C)/\SU$. For $X\in\SL(\C)$, let $\ph X\in\Hyp^3$ denote the left coset 
      \begin{equation}\label{eq:leftcoset}
        \lc X:=\lbrace XU\;\mid\;U\in\SU\rbrace.
      \end{equation}
      $M\in\SL(\C)$ acts isometrically on the hyperbolic $3$-space $\Hyp^3$ by
      \begin{equation}\label{eq:actionH3}
        \lc X\mapsto\lc{MX}.
      \end{equation}
      The fixed point set of this action is
      \begin{equation}\label{eq:fix}
        \fix(M)=\lbrace \lc X\in\Hyp^3\;\mid\;X^{-1}MX\in\SU\rbrace.
      \end{equation}
      The fixed point set $\fix(M)$, if non-empty, is called the \emph{axis} of $M$. Hyperbolic $3$-space can be extended to include the sphere at infinity as follows. Let
      \begin{equation}\label{eq:gu2}
        \GU:=\lbrace X\in\mathrm{M}_{2\times 2}\C\;\mid\;XX^*=x\mathbbm{1},\,x\neq 0\rbrace
      \end{equation}
      be the group of unitary similitudes. Let $\Nu:=\mathrm{M}_{2\times 2}\C^*$ and $\Xi:=\Nu/\GU$. Then $\Xi=A\sqcup B$ where
      \begin{equation}\label{eq:AandB}
        A:=\lbrace \lc X\in\Xi\;\mid\;\det X\neq 0\rbrace,\quad B:=\lbrace \lc X\in\Xi\;\mid\;\det X= 0\rbrace\,.
      \end{equation}
      Then $A\;=\;\GL\C/\GU\;=\;\SL(\C)/\SU\;=\;\Hyp^3$. To show $B=\C P^1$, note that any $X\in\Nu$ with $\det X=0$ can be written in the form
      \begin{equation}\label{eq:X}
        X=\begin{pmatrix} x\\ y\end{pmatrix}\begin{pmatrix} a & b\end{pmatrix}
      \end{equation}
      with $(x,y)^t,\,(a,b)\in\C^2\out 0$. The first factor $(x,y)^t$ is unique up to multiplication by an element of $\C^*$, so the map $\phi:B\to\C P^1$ given by
      \begin{equation}\label{eq:maptoCP1}
        \lc{\begin{pmatrix} x\\ y\end{pmatrix}\begin{pmatrix} a & b\end{pmatrix}}\mapsto\ph{x,y}
      \end{equation}
      is well defined, and a bijection, since $\GU$ acts transitively on $\C^2\out 0$.
    %
    %%%%%%%%%%%%%%%%%%%%%%%%%%%%%%%%%%%%%%%%%%%%%%%%%%%%%%%%%%%%%%%%%

    \subsection{Unitarisability}\label{sec:unitarisability}

      \begin{definition}\label{def:unitirred}
        For the sake of clarification and to fix notation, we say that
        \begin{itemize}
          \item $M\in\SL(\C)$ is \emph{unitarisable} if there exists $X\in\SL(\C)$ such that $X^{-1} M X\in\SU$.
          \item $M_1,\dots,M_n\in\SL(\C)$ are \emph{simultaneously unitarisable} if there exists $X\in\SL(\C)$ such that $X^{-1} M_k X\in\SU$ for $k\in\lbrace 1,\dots,n\rbrace$.
          \item $M\in\SL(\C)$ is irreducible if it is not \emph{similar} via a permutation to a block upper triangular matrix.
        \end{itemize}
      \end{definition}
      The next proposition follows immediately from (\ref{eq:fix}).
      \begin{proposition}\label{axis}
        $M\in\SL(\C)$ is unitarisable if and only if it has an axis.
      \end{proposition}
      \begin{proposition}\label{intersect}
        Individually unitarisable matrices $M_1,\dots,M_n\in\SL(\C)$ are simultaneously unitarisable if and only if their axes intersect in a common point.
      \end{proposition}
      \begin{proof}
        $M_1,\dots,M_n\in\SL(\C)$ are simultaneously unitarisable if and only if there exists $X\in\SL(\C)$ such that
        $X^{-1}M_kX\in\SU$ for all $k\in\lbrace 1,\dots,n\rbrace$. This is equivalent to $X\in\fix(M_k)$ for all $k\in\lbrace 1,\dots,n\rbrace$.
      \end{proof}
      %
    %
    %%%%%%%%%%%%%%%%%%%%%%%%%%%%%%%%%%%%%%%%%%%%%%%%%%%%%%%%%%%%%%%%%

    \subsection{Eigenlines}\label{sec:eigenlines}

      The fixed point set $\fix(M)$ of $M\in\SL(\C)\out{\pm\identity}$ is a disjoint union of a (possibly empty) component in $\Hyp^3$ and a component on the sphere at infinity:
      \begin{equation}\label{eq:fixcomponents}
        \fix(M)=(\fix(M)\cap A)\sqcup(\fix(M)\cap B).
      \end{equation}
      The part in $\Hyp^3$, if non-empty, is the \emph{axis} of $M$. The part on the sphere at infinity is the set of \emph{eigenlines} of $M$, as the following proposition shows. Note that this part consists of exactly one or two points, since $M$ has one or two eigenlines.
      \begin{proposition}\label{eigenlines}
        For $M\in\SL(\C)\out{\pm\identity}$, the set $\phi(\fix(M)\cap B)\subset\C P^1$ is the set of eigenlines of $M$.
      \end{proposition}
      \begin{proof}
        Note that $\fix(M)\cap B$ is the set of elements $X\in\mathrm{M}_{2\times 2}\C$ with $\det X = 0$ such that $\lc{MX}=\lc X$. Since $\det X = 0$, $X$ is of the form (\ref{eq:X}) with $(x,y)^t,\,(a,b)\in\C^2\out 0$. Since $\phi$ is a bijection, $\lc{MX}=\lc X$ if and only if $\phi(\lc{MX})=\phi(\lc X)$. That is, if and only if in $\C P^1$
        \begin{equation}\label{eq:Mxy}
          \lc{M\begin{pmatrix} x \\ y\end{pmatrix}}=\lc{\begin{pmatrix} x \\ y\end{pmatrix}}.
        \end{equation}
        That is, if and only if $(x,y)^t$ is an eigenline of $M$.
      \end{proof}
      %
    %
    %%%%%%%%%%%%%%%%%%%%%%%%%%%%%%%%%%%%%%%%%%%%%%%%%%%%%%%%%%%%%%%%%

    \subsection{The Klein model of $\Hyp^3$}\label{sec:kleinmodel}

      The Klein model of $\Hyp^3$ is the unit ball $ \B=\lbrace x\in\R^3 \mid \|x\|\leq 1\rbrace$. The sphere at infinity is its boundary $\partial\B=\lbrace x\in\R^3 \mid \|x\|=1\rbrace$.
      The map $K:\Hyp^3\to\overline{\B}$ is defined as the map $\lc X\mapsto XX^*$ followed by the map
      \begin{equation}\label{eq:unitarymap}
        \begin{pmatrix} a+b & c+id \\ c-id & a-b\end{pmatrix}=\frac{1}{a}(b,c,d).
      \end{equation}
      We need some well-known results about this model. For details, see \cite[Sections~II.5, VIII]{Iver} and \cite[Sections~A.4, A.5]{BenP}. Non-trivial isometries in $\Hyp^3$ are identified with elements of $\SL(\C)\out{\pm\identity}$. In particular, we have the following
      \begin{proposition}\label{elliptic}
        $M\in\SL(\C)\out{\pm\identity}$ is unitarisable if and only if $M$ is elliptic as an isometry of $\Hyp^3$.
      \end{proposition}
      \begin{proposition}\label{straight}
        If $M\in\SL(\C)\out{\pm\identity}$ is unitarisable, then the axis of $M$ in the Klein model is a Euclidean straight line segment in $\overline{\B}$ with two distinct endpoints on $\partial\B$.
      \end{proposition}
      \begin{proof}
        Since $M\in\SL(\C)$ is unitarisable, by \Cref{axis}, $fix(M)\neq\emptyset$. Following \cite[Section VIII.11]{Iver}, $M$ has two fixed endpoints on $\partial\B$ and the axis through these is a geodesic. Since geodesics are straight line segments, also the axis of $M$ is a straight line segment.
      \end{proof}
      %
    %
    %%%%%%%%%%%%%%%%%%%%%%%%%%%%%%%%%%%%%%%%%%%%%%%%%%%%%%%%%%%%%%%%%

    \subsection{The Cross Ratio}\label{sec:crossratio}

      The cross ratio of four distinct points in $\C P^1=\C\cup\lbrace\infty\rbrace$ is denoted by
      \begin{equation}\label{eq:crossratio}
        \left[a,b,c,d\right]:=\frac{(b-c)(d-a)}{(b-a)(d-c)}.
      \end{equation}
      The cross ratio is chosen so that $\left[0,1,\infty,x\right]=x$. The cross ratio is invariant under M\"obius transformations. The cross ratio $\left[a,b,c,d\right]$ of four distinct points is real if and only if the points lie on a circle.

      \begin{proposition}\label{cratio}
        Let $a$, $b$, $c$ and $d$ be distinct points in $\C P^1$, and suppose $[a,b,c,d]\in\R\out 0$, so $a$, $b$, $c$, $d$ lie on a circle $\mathcal{C}$.
        \begin{enumerate}
          \item $\left[a,b,c,d\right]\in\R_+$ if and only if $b$ and $d$ lie in the same connected component of $\mathcal{C}\out{a,c}$.
          \item $\left[a,b,c,d\right]\in\R_-$ if and only if $b$ and $d$ lie on different connected components of $\mathcal{C}\out{a,c}$.
        \end{enumerate}
      \end{proposition}
      \begin{proof}
        There exists a unique M\"obius transformation taking $a$, $b$ and $c$ to $0$, $1$ and $\infty$ respectively, taking circles to circles, and preserving or reversing the order of $a$, $b$, $c$ and $d$ on the circle. Thus we may assume $a=0$, $b=1$, $c=\infty$. Thus $\left[a,b,c,d\right]=\left[0,1,\infty,d\right]=d\in\R\out{0,1}$. Then $a$, $b$, $c$ and $d$ lie on the circle $\mathcal{C}=\R\cup\lbrace\infty\rbrace$, and the two connected components of $\mathcal{C}\out{a,c}$ are $\R_-$ and $\R_+$. The theorem follows by an examination of the two cases $d\in\R_+\out 1$ and $d\in\R_-$.
      \end{proof}
      %
    %
    %%%%%%%%%%%%%%%%%%%%%%%%%%%%%%%%%%%%%%%%%%%%%%%%%%%%%%%%%%%%%%%%%

    \subsection{Unitarisability of two matrices}\label{sec:unit2matrices}

      We conclude by proving \Cref{unit2mat}, which gives a characterization for the simultaneous unitarisability of two matrices in $\SL(\C)$.
      
      {\emph{Proof of \Cref{unit2mat}}}.\,
      Since $M_0$ and $M_1$ are individually unitarisable then, by \Cref{axis}, $\fix(M_0)\neq\emptyset$ and $\fix(M_1)\neq\emptyset$. By \Cref{intersect}, $M_0$ and $M_1$ are simultaneuosly unitarisable if and only if $\fix(M_0)\cap\fix(M_1)\neq\emptyset$.

      First suppose that $\fix(M_0)\cap\fix(M_1)\neq\emptyset$. By \Cref{straight}, $\fix(M_0)$ and $\fix(M_1)$ are straight line segments. And since they intersect, they lie in a unique Euclidean plane $\mathcal{P}\subset\R^3$. Then $\mathcal{C}:=\mathcal{P}\cap\partial\B$ is a circle. By \Cref{eigenlines}, the endpoints of $\fix(M_0)$ and $\fix(M_1)$ are $\varphi, \varphi'$ and $\psi, \psi'$ respectively. Since $\fix(M_0)$ and $\fix(M_1)$ intersect at a point inside the disk $\mathcal{P}\cap\B$ bounded by $\mathcal{C}$, it follows that $\psi$and $\psi'$ lie in different connected components of $\mathcal{C}\out{\varphi, \varphi'}$. By \Cref{cratio}, $\left[\varphi, \psi, \varphi', \psi'\right]\in\R_-$.

      Conversely, suppose $\left[\varphi, \psi, \varphi', \psi'\right]\in\R_-$. By \Cref{cratio} $\varphi, \psi, \varphi'$ and $\psi'$ lie on some circle $\mathcal{C}\subset\partial\B$, and $\psi$ and $\psi'$ lie on different connected components of $\mathcal{C}\out{\varphi, \varphi'}$. Let $\mathcal{P}$ be the unique Euclidean plane containing $\mathcal{C}$. By \Cref{eigenlines} and \Cref{straight}, $\fix(M_0)$ is the straight line segment with endpoints $\varphi$ and $\varphi'$, and $\fix(M_1)$ is the straight line segment with endpoints $\psi$ and $\psi'$. Hence $\fix(M_0)$ and $\fix(M_1)$ lie on $\mathcal{P}$ and intersect in the disk $\mathcal{P}\cap\B$ bounded by $\mathcal{C}$. Thus $\fix(M_0)\cap\fix(M_1)\neq\emptyset$. \hfill $\Box$
      %
    %
    %%%%%%%%%%%%%%%%%%%%%%%%%%%%%%%%%%%%%%%%%%%%%%%%%%%%%%%%%%%%%%%%%
  %
  %%%%%%%%%%%%%%%%%%%%%%%%%%%%%%%%%%%%%%%%%%%%%%%%%%%%%%%%%%%%%%%%%%%%%%%%%

  \bibliographystyle{amsplain}

\begin{thebibliography}{10}

    \bibitem{Heun}
    F.~M. Arscott and S.~Yu. Slavyanov, \emph{Heun's differential equations},
      Oxford University Press, 1995, ed. A. Ronveaux.

    \bibitem{BenP}
    R.~Benedetti and C.~Petronio, \emph{Lectures on hyperbolic geometry}, Springer,
      1992.

    \bibitem{Bob}
    A.~I. Bobenko, \emph{Surfaces in terms of 2 by 2 matrices. {O}ld and new
      integrable cases}, Harmonic maps and integrable systems, Aspects Math., E23,
      Vieweg, Braunschweig (1994), 83--127.

    \bibitem{DPW}
    J.~Dorfmeister, F.~Pedit, and H.~Wu, \emph{Weierstrass type representation of
      harmonic maps into symmetric spaces}, Comm. Anal. Geom. \textbf{6} (1998),
      633--668.

    \bibitem{KGB}
    K.~Gro\ss e~Brauckmann, R.B. Kusner, and J.M. Sullivan, \emph{Triunduloids:
      embedded constant mean curvature surfaces with three ends and genus zero}, J.
      Reine Angew. Math. \textbf{564} (2003), 35--61.

    \bibitem{Iver}
    B.~Iversen, \emph{Hyperbolic geometry}, Cambridge University Press, 1992.

    \bibitem{KilRS}
    M.~Kilian, W.~Rossman, and N.~Schmitt, \emph{Delaunay ends of constant mean
      curvature surfaces}, Compos. Math. \textbf{144} (2009), no.~1, 186--220.

    \bibitem{KilSch}
    M.~Kilian and N.~Schmitt, \emph{Constant mean curvature cylinders with
      irregular ends}, J. Math. Soc. Japan \textbf{65} (2013), 775--786.

    \bibitem{KKS}
    N.~Korevaar, R.~Kusner, and B.~Solomon, \emph{The structure of complete
      embedded surfaces with constant mean curvature}, J. Differential Geom.
      \textbf{30} (1989), 465--503.

    \bibitem{ScSc}
    R.~Sch{\"a}fke and D.~Schmidt, \emph{The connection problem for general linear
      ordinary differential equations at two regular singular points with
      applications in the theory of special functions}, SIAM Journal on
      Mathematical Analysis \textbf{11} (1980), 848--862.

    \bibitem{Sch}
    N.~Schmitt, \emph{{CMCL}ab, \url{http://www.gang.umass.edu/software}}.

    \bibitem{KilKRS}
    N.~Schmitt, M.~Kilian, S.-P. Kobayashi, and W.~Rossman, \emph{Unitarization of
      monodromy representations and constant mean curvature trinoids in
      3-dimensional space forms}, J. Lond. Math. Soc. \textbf{75} (2007), no.~3,
      563--581.

  \end{thebibliography}

  \providecommand{\bysame}{\leavevmode\hbox to3em{\hrulefill}\thinspace}
  \providecommand{\MR}{\relax\ifhmode\unskip\space\fi MR }
  \providecommand{\MRhref}[2]{%
    \href{http://www.ams.org/mathscinet-getitem?mr=#1}{#2}
  }
  \providecommand{\href}[2]{#2}

\end{document}